\documentclass[11pt]{amsart}
\usepackage[margin=3.15cm]{geometry}   

\usepackage{microtype}
\usepackage{caption}
\usepackage{mathrsfs}
\usepackage{graphicx}
\usepackage[all,pdftex,arc,curve,color,frame]{xy}
\usepackage{hyperref}
\usepackage[lite]{amsrefs}
\usepackage{amsmath}
\usepackage{amssymb}
\usepackage{xcolor,colortbl}

\newcommand{\R}{\mathbb R}

\newcommand{\Z}{\mathbb Z} 
\newcommand{\F}{\mathcal F}
\newcommand{\T}{\mathbf{T}}
\newcommand{\bL}{\mathbb L}

\newcommand{\al}{\alpha}
\newcommand{\be}{\beta}

\newcommand{\ep}{\varepsilon}
\newcommand{\si}{\sigma}

\newcommand{\Ga}{\Gamma}
\newcommand{\La}{\Lambda}

\newcommand{\om}{\omega}

\newcommand{\CP}{\mathbb{CP}^2}
\newcommand{\CPl}{\mathbb{CP}^1}
\newcommand{\CPb}{\overline{\mathbb{CP}}\vphantom{\CP}^2}

\newcommand{\del}{\partial}

\newcommand{\lra}{\longrightarrow}

\DeclareMathOperator{\coker}{coker}

\DeclareMathOperator{\trace}{tr}
\DeclareMathOperator{\lk}{lk}
\DeclareMathOperator{\tb}{tb}
\DeclareMathOperator{\rot}{rot}

\newtheorem{thm}{Theorem}[section]
\newtheorem{lemma}[thm]{Lemma}

\newtheorem{prop}[thm]{Proposition}
\newtheorem{conj}{Conjecture}

\begin{document}

\author{Marco Golla and Paolo Lisca}
\title[On Stein fillings of contact torus bundles]
{On Stein fillings of contact torus bundles}
\address{Dipartimento di Matematica, Universit\`a di Pisa, Largo Bruno Pontecorvo 5, 56121 Pisa, Italy}
\email{marco.golla@for.unipi.it, lisca@dm.unipi.it}

\begin{abstract}
We consider a large family $\F$ of torus bundles over the circle, and we use recent 
work of Li--Mak to construct, on each $Y\in\F$, a Stein fillable contact structure $\xi_Y$. 
We prove that (i) each Stein filling  of $(Y,\xi_Y)$ has vanishing first Chern class and first Betti number, (ii) if $Y\in\F$ is elliptic then all Stein fillings of $(Y,\xi_Y)$ are pairwise diffeomorphic and (iii) if $Y\in\F$ is parabolic or hyperbolic then all Stein fillings of $(Y,\xi_Y)$ share the same Betti numbers and fall into finitely many diffeomorphism classes. Moreover, for infinitely 
many hyperbolic torus bundles $Y\in\F$ we exhibit non-homotopy equivalent Stein fillings
of $(Y,\xi_Y)$.
\end{abstract}

\maketitle

\section{Introduction}\label{s:intro}
The diffeomorphism classification of symplectic fillings has been previously considered by several authors. 
For the standard definitions on symplectic structures, contact structures and their symplectic fillings we refer the reader to~\cites{Ge08, McDS98}. 
The first classification result for symplectic fillings is due to Eliashberg~\cite{El91}, 
who proved that a symplectic filling of the standard contact $S^3$ is diffeomorphic to a blowup
of $B^4$. McDuff~\cite{McD90} extended Eliashberg's result to the lens
spaces $L(p,1)$ endowed with their standard contact structures. Ohta and Ono~\cites{OO03, OO05}
determined the diffeomorphism types of symplectic fillings of links of simple 
elliptic and simple singularities endowed with their natural contact structures. Stein fillings up to diffeomorphisms were classified by the second author~\cite{Li08} for all lens spaces 
with their standard contact structures, by Plamenevskaya--Van Horn-Morris~\cite{PVH10} on $L(p,1)$ with other contact structures  
 and by Starkston~\cite{St15} for certain contact Seifert fibered 3--manifolds. 
In this paper we study Stein and symplectic fillings of infinitely many contact torus bundles over the circle. 

Below we define a large family $\F$ of closed, oriented torus bundles over $S^1$, and in Section~\ref{s:caps} 
we use recent work of Li--Mak~\cite{LM14} to construct a Stein fillable contact structure $\xi_Y$ for each $Y\in\F$. 
The following is our main result.
\begin{thm}\label{t:main}
Let $Y\in\F$. Then, each Stein filling  of $(Y,\xi_Y)$ has vanishing first Chern class and first Betti number. If $Y$ is elliptic then $(Y,\xi_Y)$ admits a unique Stein filling up to diffeomorphisms. 
If $Y$ is parabolic or hyperbolic then all the Stein fillings of $(Y,\xi_Y)$ share the same Betti numbers and fall into finitely 
many diffeomorphism classes. 
\end{thm}
As shown in Theorems~\ref{t:ell1-hyp} and~\ref{t:par}, for some elliptic bundles and for parabolic and hyperbolic bundles in $\F$ the results of Theorem~\ref{t:main} hold more generally for minimal, strongly convex symplectic fillings rather than just for Stein fillings. 

We are now going to describe the family $\F$. 
We will denote by $\T_A$ an oriented torus bundle over $S^1$ with monodromy specified by a matrix $A\in SL_2(\Z)$. It is a well-known
fact (cf.~\cite{Ne81}*{Lemma~6.2}) that $\T_A$ is orientation-preserving diffeomorphic to $\T_B$ if and only if $A$ is conjugate 
in $SL_2(\Z)$ to $B$. Moreover, $-\T_A$ is orientation-preserving diffeomorphic to $\T_{A^{-1}}$. 
A torus bundle $\T_A$ is called \textbf{elliptic} if $|\trace A|<2$, \textbf{parabolic} if $|\trace(A)|=2$ 
and \textbf{hyperbolic} if $|\trace(A)|>2$. Given $(d_1,\dots,d_m)\in\Z^m$, $m\geq 1$, we define 
\begin{equation}\label{e:A(d)}
A(d_1,\ldots, d_m) := 
\begin{pmatrix}
d_m & 1\\
-1 & 0
\end{pmatrix}
\cdots
\begin{pmatrix}
d_1 & 1\\
-1 & 0
\end{pmatrix}\in SL_2(\Z).
\end{equation}
By~\cite{Ne81}*{Proposition~6.3}, if two $m$--tuples $d, d'\in\Z^m$ as above are obtained from each other by a cyclic permutation 
then $\T_{A(d)} = \T_{A(d')}$ and, by~\cite{Ne81}*{Theorem~6.1} $\T_A$ is hyperbolic with $\trace(A)<-2$ 
(respectively $\trace(A)>2$) if and only if $\T_A = \T_{-A(d)}$ (respectively $\T_A = \T_{A(d)}$) for some $d=(d_1,\dots,d_m)$ with 
$d_i\geq 2$ for all $i$ and $d_i\geq 3$ for some $i$. Moreover, by the proof of~\cite{Ne81}*{Theorem~6.1} 
and~\cite{Ne81}*{Theorem~7.3}, if 
 \[
d = (n_1+3,\underbrace{2,\dots,2}_{m_1},n_2+3,\underbrace{2,\dots,2}_{m_2},\dots,n_\ell + 3, \underbrace{2,\dots,2}_{m_\ell}), \quad 
m_i, n_i\geq 0,
\]
then $-\T_{\pm A(d)} = \T_{\pm A(\rho(d))}$, where 
\begin{equation}\label{e:rho}
\rho(d) := (m_1+3,\underbrace{2,\dots,2}_{n_1},m_2+3,\underbrace{2,\dots,2}_{n_2},\dots,m_\ell + 3, \underbrace{2,\dots,2}_{n_\ell}).
\end{equation}

The following definition is inspired by a similar definition from~\cite{Li13}.  
A \textbf{blowup} of a sequence $(s_1,\dots,s_\ell)$ of nonnegative integers is one of the following sequences:
\[
(s_1,\dots,s_{i-1},s_{i}+1,1,s_{i+1}+1,s_{i+2},\dots,s_\ell), \quad i=1,\ldots, \ell-1.
\]
If a sequence $s'$ is obtained from the sequence $s$ through a finite number of blowups, we also say that $s'$ is a 
blowup of $s$. Given two sequences $s, c$ of length $\ell$, we write $s\prec c$ if $s_i\le c_i$ for every $1\le i\le \ell$, 
and we say that $d\in\Z^m$ is \textbf{embeddable} if $s\prec\rho(d)$ for some blowup $s$ of $(0,0)$. 

We define $\F$ to be the set of torus bundles $Y$ over the circle such that one of following holds: 
\begin{enumerate} 
\item 
$Y$ is elliptic; 
\item
$Y$ is parabolic and $Y = -\T_{A(0,-n)}$ with $n\leq 4$; 
\item
$Y$ is hyperbolic and $Y=-\T_{A(-c)}$ with $c\geq 3$;
\item 
$Y$ is hyperbolic and $Y= \T_{-A(d)}$ with $d$ embeddable. 
\end{enumerate} 

Combining Proposition~\ref{p:ell-class} and Theorem~\ref{t:hyp-class} we obtain the following.

\begin{thm}\label{t:class}
Let $Y$ be a torus bundle of type $\T_{-A(\ep)}$ with $\ep\in\{-1,0,1\}$, $\T_{A(1)}$ or 
$\T_{-A(d)}$ with $d$ embeddable. 
Then, the contact structure $\xi_Y$ is the unique universally tight contact structure on 
$Y$ with vanishing Giroux torsion. 
\end{thm} 

Theorem~\ref{t:class} has led us to formulate Conjecture~\ref{conj} below. Before we can state it we need to introduce some notation. 
Let $(W,\om)$ be a symplectic 4--manifold. A collection $D = C_1\cup\cdots\cup C_n$ of 
finitely many closed, embedded, symplectic surfaces in $W$ intersecting transversely 
and positively, and such that no three of them have a point in common will be called a \textbf{symplectic divisor}. 
When the symplectic form $\om$ is part of a K\"ahler structure on $W$ and the surfaces $C_i$ are 
smooth, complex curves, we will call $D$ a~\textbf{complex divisor}. 
When each $C_i$ is a 2--sphere the divisor will be called~\textbf{spherical}. 

\begin{conj}\label{conj}
Let $(X,\om)$ be a closed symplectic 4--manifold obtained as a symplectic blowup of $\CP$ with 
the standard K\"ahler form. Suppose that  
\[
D=C_1\cup\cdots\cup C_n\subset X 
\]
is a circular, spherical symplectic divisor such that $C_i\cdot C_i\in\{0,+1\}$ for some $i\in\{1,\ldots,n\}$. Then,
any contact structure induced on the boundary of a concave neighbourhood of $D$ is universally tight.
\end{conj}

The paper is organized as follows. In Section~\ref{s:caps} we use the work of Li--Mak~\cite{LM14} to prove Theorem~\ref{t:caps}, which says that, given a bundle $Y\in\F$, there exists a compact, symplectic 4--manifold with strictly $\om$--concave 
boundary $(W_Y,\om_Y)$ such that $\del W_Y = -Y$ and $(W_Y,\om_Y)$ embeds 
symplectically in a (deformation of) a blowup of the complex projective plane. 
The symplectic 4--manifolds $(W_Y,\om_Y)$ are used in Section~\ref{s:fillings} to classify, up to diffeomorphisms, 
the Stein fillings of the contact 3--manifolds $(Y,\xi_Y)$, where $\xi_Y$ is the positive contact structure on 
$Y$ induced by the $\om$--concave structure on the boundary of $W_Y$. Theorem~\ref{t:main} follows combining Theorems~\ref{t:ell1-hyp}, \ref{t:ell2} and~\ref{t:par}. In Section~\ref{s:fillings} we also prove Proposition~\ref{p:distfill}, showing the existence of infinitely many hyperbolic torus bundles $Y\in\F$ such that $(Y,\xi_Y)$ admits non-homotopy equivalent Stein fillings. In Section~\ref{s:contstr} we identify the contact structures $\xi_Y$ for some elliptic and hyperbolic bundles by proving Proposition~\ref{p:ell-class} and Theorem~\ref{t:hyp-class}, which imply Theorem~\ref{t:class}. We also give explicit constructions of Stein fillings for $(Y,\xi_Y)$ when $Y$ 
is an elliptic torus bundle of type $\T_{-A(\ep)}$ with $\ep\in\{-1,0,1\}$, or $\T_{A(1)}$.

\subsection*{Acknowledgements} 
The authors wish to thank Youlin Li for suggesting that Theorem~\ref{t:par}, originally proved for Stein 
fillings, might hold more generally for minimal, strongly convex symplectic fillings. 
The present work is part of the authors' activities within CAST, a Research Network
Program of the European Science Foundation. Both authors were partially supported by 
the PRIN--MIUR research project 2010--11 ``Variet\`a reali e complesse: geometria, topologia e analisi armonica'', 
the first author was partially supported by the FIRB research project "Topologia e
geometria di variet\`a in bassa dimensione" and by an ERC Exchange Grant.

\section{Construction of symplectic caps}\label{s:caps}

In this section we prove that for each torus bundle $Y$ belonging to the family $\F$ of Section~\ref{s:intro}, 
there exists  a compact, symplectic 4--manifold with strictly $\om$--concave 
boundary $(W_Y,\om_Y)$ such that $\del W_Y = -Y$ and $(W_Y,\om_Y)$ embeds 
symplectically in a (deformation of) a symplectic blowup of the standard symplectic $\CP$. 
We call a symplectic 4--manifold $(W_Y,\om_Y)$ as above a \textbf{symplectic cap} of $Y$. 
Our main tool to construct the symplectic 4--manifolds 
$W_Y$ will be the following theorem by Li and Mak. 
\begin{thm}[{\cite{LM14}*{Theorem 1.3}}]\label{t:limak}
Let $D\subset (W,\om_0)$ be a symplectic divisor. If the intersection form of $D$ is not negative definite and the restriction of $\om_0$ to the boundary of a closed regular neighborhood of $D$ is exact, then $\om_0$ can be deformed through a family of symplectic forms $\om_t$ on $W$ keeping $D$ symplectic and such that, for any neighborhood $N$ of $D$, there is an $\om_1$--concave neighborhood of $D$ inside $N$. 
\end{thm} 
In the proof of Theorem~\ref{t:caps} we will apply
Theorem~\ref{t:limak} to certain suitable spherical complex divisors in blowups
of the complex plane $\CP$ endowed with their standard K\"ahler structure. 
We will obtain the divisors that we need by blowing up the following 
two basic configurations of immersed complex spheres in $\CP$: 
\begin{enumerate}
\item[$(3\ell)$] 
three complex lines in general position;
\item[$(\ell{\mathcal C}_2)$] 
a line and a smooth conic in general position.
\end{enumerate}
Regular neighborhoods of  Configurations $(3\ell)$ and $(\ell{\mathcal C}_2)$ are 4--dimensional plumbings given, 
in the notation of Neumann~\cite{Ne81}, by the graphs of Figure~\ref{f:basic_confs}.
\begin{figure}[ht]
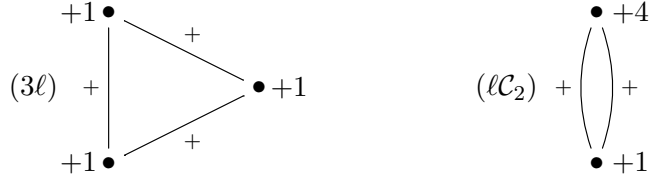

\centering
$\xygraph{
!{<0cm,0cm>;<1cm,0cm>:<0cm,1cm>::}
!{(2,1) }*+{\bullet}="a"
!{(2.4,1) }*+{+1}
!{(0,2) }*+{\bullet}="b"
!{(-0.4,2) }*+{+1}
!{(0,0) }*+{\bullet}="c"
!{(-0.4,0) }*+{+1}
!{(-1,1) }*+{(3\ell)}
"a"-"b"_+
"b"-"c"_+
"a"-"c"^+
}
\hspace{2cm}
\xygraph{
!{<0cm,0cm>;<1cm,0cm>:<0cm,1cm>::}
!{(-0.2,0) }*+{\bullet}="a"
!{(-0.2,2) }*+{\bullet}="b"
!{(0.25,2) }*+{+4}
!{(0.25,0) }*+{+1}
!{(-1.4,1) }*+{(\ell{\mathcal C}_2)}
"a"-@/^/"b"^+
"a"-@/_/"b"_+}$
\caption{Plumbing graphs of the two basic configurations}
\label{f:basic_confs}
\end{figure}  

\subsection*{Elliptic bundles} 
Let $\T_A$ be a torus bundle with $|\trace(A)|<2$. It follows from the proof of~\cite{Ne81}*{Proposition~2.1} 
(see~\cite{Ne81}*{page 307}) that there are exactly 
six such torus bundles up to orientation-preserving diffeomorphisms, i.e.~$\T_{\pm A(\ep)}$, with $\ep = -1,0,1$ (here we are using 
Notation~\eqref{e:A(d)}). We claim that these bundles are the oriented boundaries of the six 
4--dimensional plumbings given by Figure~\ref{f:ell_bundles}. 
\begin{figure}[ht]
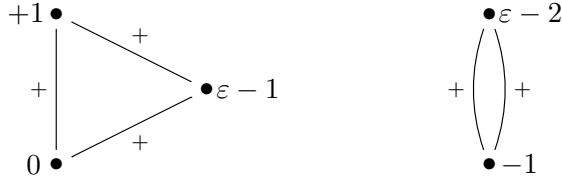

\centering
$\xygraph{
!{<0cm,0cm>;<1cm,0cm>:<0cm,1cm>::}
!{(2,1) }*+{\bullet}="a"
!{(2.55,1) }*+{\ep-1}
!{(0,2) }*+{\bullet}="b"
!{(-0.4,2) }*+{+1}
!{(0,0) }*+{\bullet}="c"
!{(-0.3,0) }*+{0}
"a"-"b"_+
"b"-"c"_+
"a"-"c"^+
}
\hspace{2cm} 
\xygraph{
!{<0cm,0cm>;<1cm,0cm>:<0cm,1cm>::}
!{(-0.2,0) }*+{\bullet}="a"
!{(-0.2,2) }*+{\bullet}="b"
!{(0.35,2) }*+{\ep-2}
!{(0.2,0) }*+{-1}
"a"-@/^/"b"^+
"a"-@/_/"b"_+
}$
\caption{Plumbing graphs for elliptic torus bundles, $\ep=-1,0,1$.}
\label{f:ell_bundles}
\end{figure} 
Indeed, the proof of~\cite{Ne81}*{Theorem~6.1} shows that the bundle given by the graph on 
the left of Figure~\ref{f:ell_bundles} has monodromy 
\[
A(1-\ep,0,-1) = 
\begin{pmatrix} 
\ep & -1\\
1 & 0
\end{pmatrix}
= -A(-\ep),
\]
while the monodromy of the bundle given by the graph on the right is  
\[
A(1,2-\ep) = 
\begin{pmatrix} 
1-\ep & -\ep+2\\
-1 & -1
\end{pmatrix}
=
\begin{pmatrix} 
1 & -1\\
0 & 1
\end{pmatrix} 
A(-\ep)
\begin{pmatrix} 
1 & -1\\
0 & 1
\end{pmatrix} ^{-1}.
\]

\begin{lemma}\label{l:ell_sd}
For $\ep\in\{-1,0,1\}$ the graph on the left-hand side of Figure~\ref{f:ell_bundles} is dual to 
the intersection graph of a spherical complex divisor $D\subset \CP\#(3-\ep)\CPb$, while 
the graph on the right-hand side of Figure~\ref{f:ell_bundles} is dual to the intersection graph 
of a spherical complex divisor $D\subset\CP\#(8-\ep)\CPb$.
\end{lemma} 

\begin{proof} 
(1) Let the $\ell_1, \ell_2, \ell_3$ be the three generic lines of the basic configuration $(3\ell)$ inside $\CP$ with 
the Fubini--Study form. Blow up $\CP$ at one generic point of $\ell_2$ and at $2-\ep$ generic points 
of $\ell_3$, and let $D\subset\CP\#(3-\ep)\CPb$ be the proper transform of $\ell_1\cup\ell_2\cup\ell_3$. 

(2) $D$ is obtained as the proper transform of the configuration $\ell\mathcal{C}_2$ in $\CP$ blown up at 
two generic points of the line and at $6-\ep$ generic points of the conic.
\end{proof} 

\subsection*{Parabolic bundles}
Arguing as in the proof of~\cite{Ne81}*{Theorem~6.1} and using Notation~\eqref{e:A(d)} 
it is easy to check that the boundary of the plumbing given by the graph of 
Figure~\ref{f:par_bundles} is a (parabolic) torus bundle with monodromy 
$A(0,-n)=-\left(\begin{smallmatrix}1&n\\0&1\end{smallmatrix}\right)$. 
\begin{figure}
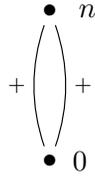

\centering
$\xygraph{
!{<0cm,0cm>;<1cm,0cm>:<0cm,1cm>::}
!{(-0.2,0) }*+{\bullet}="a"
!{(-0.2,2) }*+{\bullet}="b"
!{(0.3,2) }*+{n}
!{(0.2,0) }*+{0}
"a"-@/^/"b"^+
"a"-@/_/"b"_+
}$
\caption{Plumbing graphs for parabolic bundles, $n\in\Z$.}
\label{f:par_bundles}
\end{figure}

\begin{lemma}\label{l:par_sd}
For every integer $n\leq 4$ the graph of Figure~\ref{f:par_bundles} is dual to the 
intersection graph of a spherical complex divisor $D\subset\CP\#(5-n)\CPb$.
\end{lemma} 

\begin{proof} 
When $n\leq 4$ the graph of Figure~\ref{f:par_bundles} is the intersection graph of the proper transform 
of the basic configuration $(\ell{\mathcal C}_2)$ in $\CP$, obtained by blowing up at $4-n$ generic points  
of the conic ${\mathcal C}$ and one generic point of the line $\ell$.
\end{proof} 

\subsection*{Hyperbolic bundles}
Let $\T_A$ by a hyperbolic bundle with $\trace(A)< -2$. As explained in Section~\ref{s:intro}, 
$\T_A=\T_{-A(d)}$, where $d=(d_1,\ldots, d_m)\in\Z^m$, $d_i\geq 2$ for all $i$ and $d_i\geq 3$ 
for some $i$, and $-\T_{-A(d)} = \T_{-A(\rho(d))}$, where $\rho(d) = (c_1,\ldots, c_\ell)$ is defined by Equation~\eqref{e:rho}. Moreover, by~\cite{Ne81}*{Theorem~7.1} $\T_{-A(\rho(d))}$ is the 
boundary of the 4--dimensional plumbing given by Figure~\ref{f:hyp_bundles1}.
\begin{figure}[ht]
\centering
$ \xygraph{
!{<0cm,0cm>;<1cm,0cm>:<0cm,1cm>::}
!{(0,1) }*+{\bullet}="c"
!{(1.6,1) }*+{\bullet}="d"
!{(3.2,1) }*+{\bullet}="e"
!{(5.6,1) }*+{\bullet}="f"
!{(7.2,1) }*+{\bullet}="g"
!{(0,1.4) }*+{-c_1}
!{(1.6,1.4) }*+{-c_2}
!{(3.2,1.4) }*+{-c_3}
!{(5.6,1.4) }*+{-c_{\ell-1}}
!{(7.2,1.4) }*+{-c_\ell}
!{(4.4,1)}*+{\cdots}
"c"-"d"^+
"d"-"e"^+
"f"-"g"^+
"g"-@/^0.5cm/"c"^{-}
}$
\caption{Plumbing graphs for $-\T_A$ with $\T_A$ hyperbolic and $\trace(A)\leq -3$.}
\label{f:hyp_bundles1}
\end{figure} 
Using Neumann's plumbing calculus (i.e.~\cite{Ne81}*{Proposition~2.1}) it is easy to check that when $\ell>1$  
the bundle $-\T_A = \T_{-A(\rho(d))}$ is also the oriented boundary of the plumbing 
given by the graph on the left of Figure~\ref{f:hyp_bundles2}, while when $\ell=1$ it is  
given by the graph 
on the right of the same figure (observe that in this case $c_1\geq 3$ by~\eqref{e:rho}). 
\begin{figure}[ht]
\centering
$ \xygraph{
!{<0cm,0cm>;<1cm,0cm>:<0cm,1cm>::}
!{(0,1) }*+{\bullet}="c"
!{(1.6,1) }*+{\bullet}="d"
!{(3.2,1) }*+{\bullet}="e"
!{(5.6,1) }*+{\bullet}="f"
!{(7.2,1) }*+{\bullet}="g"
!{(0,1.4) }*+{+1}
!{(1.6,1.4) }*+{-c_1+1}
!{(3.2,1.4) }*+{-c_2}
!{(5.6,1.4) }*+{-c_{\ell-1}}
!{(7.2,1.4) }*+{-c_\ell + 1}
!{(4.4,1)}*+{\cdots}
"c"-"d"^+
"d"-"e"^+
"f"-"g"^+
"g"-@/^0.5cm/"c"^+
}$
\hspace{1cm} 
$\xygraph{
!{<0cm,0cm>;<1cm,0cm>:<0cm,1cm>::}
!{(-0.2,0) }*+{\bullet}="a"
!{(-0.2,2) }*+{\bullet}="b"
!{(0.55,2) }*+{-c_1+2}
!{(0.2,0) }*+{+1}
"a"-@/^/"b"^+
"a"-@/_/"b"_+
}$
\caption{Alternative plumbing graphs for $-\T_A$ with $\trace(A)\leq -3$.}
\label{f:hyp_bundles2}
\end{figure} 

\begin{lemma}\label{l:hyp_sd}
Let $d=(d_1,\ldots, d_m)\in\Z^m$ with $d_i\geq 2$ for all $i$, $d_i\geq 3$ 
for some $i$ and let $(c_1,\ldots, c_\ell) = \rho(d)$.  Suppose that either $d$ is embeddable or $\ell=1$. 
Then, there is a spherical complex divisor in a blowup of $\CP$ whose dual intersection graph 
equals the graph on the left of Figure~\ref{f:hyp_bundles2} if $d$ is embeddable and the graph on the right 
of the same figure if $\ell=1$. 
\end{lemma}
 
\begin{proof} 
When $\ell=1$ the graph on the right of Figure~\ref{f:hyp_bundles2} is dual to a spherical complex divisor  
$D\subset\CP\#(c_1+2)\CPb$ consisting of the proper transforms of the line and the conic 
of the basic configuration $(\ell\mathcal{C}_2)$. This is easily shown as in the proof of Lemma~\ref{l:par_sd}. 

When $d$ is embeddable the graph on the left of Figure~\ref{f:hyp_bundles2} is dual to a 
spherical complex divisor inside a blowup of $\CP$ consisting of the proper transforms 
of the three lines of the basic configuration $(3\ell)$. In order to see this, consider 
the configuration $D_0 = (3\ell)$ of three lines in general position in $\CP$, and 
let $\ell$ be one of the lines. The line $\ell$ 
will correspond to the sphere with self-intersection $+1$ in the final divisor. Associate the string 
$s^0 = (0,0)$ to $D_0$ and define inductively configurations $D_k$, $k\geq 0$, as follows. By assumption  
there is a sequence of blowups $s^0 \leadsto s^1 \leadsto \dots \leadsto s^n = s$.  For each $k=0,\ldots, n-1$ 
the blowup $s^k\leadsto s^{k+1}$ determines in a natural way a symplectic blowup at a nodal point of 
$D_k$ not lying on $\ell$. We define $D_{k+1}$ as the 
total transform of $D_k$. The self-intersection of the $i$--th sphere in $D_k$ is $-s^k_i$, 
except for $i = 1$ and $i=k+1$, in which case the self intersection is $-s^k_i+1$. 
Finally, for each $i$ we blow up $c_i-s_i$ times at generic 
points of the $i$--th component of $D_n$ and take the resulting proper transform.
\end{proof} 

\subsection*{Existence of the symplectic caps}
We are now ready to apply Theorem~\ref{t:limak} in order to establish the following theorem, which is the 
main result of this section. 

\begin{thm}\label{t:caps}
Let $Y$ be a torus bundle over $S^1$.  Then, $Y$ admits a symplectic cap $W_Y$ which is a 
closed regular neighborhood of a spherical complex divisor $D$ in a deformation of a 
blowup of $\CP$ with its standard K\"ahler form, if one of the following conditions is verified: 
\begin{enumerate} 
\item \label{caps-item1}
$Y$ is elliptic and $Y=\T_{A(\ep)}$, with $\ep\in\{-1,0,1\}$; in this case $D$ has intersection graph dual 
to the graph on the left of Figure~\ref{f:ell_bundles};
\item \label{caps-item2}
$Y$ is elliptic and $Y=\T_{-A(\ep)}$, with $\ep\in\{-1,0,1\}$; in this case $D$ has intersection graph dual 
to the graph on the right of Figure~\ref{f:ell_bundles};
\item \label{caps-item3}
$Y$ is parabolic and $Y = -\T_{A(0,-n)}$ 
with $n\leq 4$; in this case $D$ has intersection graph dual to the graph of Figure~\ref{f:par_bundles};
\item \label{caps-item4}
$Y$ is hyperbolic and $Y=-\T_{A(-c)}$ with $c\geq 3$; in this case $D$ has intersection graph dual the graph 
on the right of Figure~\ref{f:hyp_bundles2};
\item \label{caps-item5}
$Y$ is hyperbolic $Y= \T_{-A(d)}$ with $d$ embeddable; in this case $D$ has intersection graph dual to 
the graph on the left of Figure~\ref{f:hyp_bundles2}.
\end{enumerate} 
Moreover, for each of the bundles $Y$ specified above we have $b_1(Y)=1$, and the contact 3--manifold $(Y,\xi_Y)$ admits a Stein 
filling diffeomorphic to the complement of a regular neighborhood of the corresponding spherical symplectic 
divisor $D$ constructed in one of Lemmas~\ref{l:ell_sd},~\ref{l:par_sd} or~\ref{l:hyp_sd}.
\end{thm}

\begin{proof}
We would like to apply Theorem~\ref{t:limak} to the complex divisors $D$ appearing in Lemmas~\ref{l:ell_sd},~\ref{l:par_sd} and~\ref{l:hyp_sd}. Recall that $D$ is contained in a blowup $X$ of the standard K\"ahler $\CP$.
It is easy to check using e.g.~the statement of~\cite{Ne81}*{Proposition~2.1} 
that $-\T_{A(\ep)} = \T_{-A(-\ep)}$ for each $\ep\in\{-1,0,+1\}$. Thus, in view of the three lemmas 
and the discussions preceding them, to apply Theorem~\ref{t:limak} it suffices to show that (i) the restriction of the K\"ahler form 
$\om_0$ to the boundary of a closed regular neighborhood of $D$ is exact and (ii) for each graph $\Ga$ mentioned in the statement  the corresponding intersection matrix $Q_\Ga$ is not negative definite. 
Viewing $\T_A$ as the union of two copies of a 2--torus times an interval and applying Mayer--Vietoris yields the 
exact sequence
\[
\cdots\lra \Z^2\oplus\Z^2
\stackrel{\left(\begin{smallmatrix} I & I\\A & I\end{smallmatrix}\right)}{\lra}
\Z^2\oplus\Z^2
\lra
H_1(\T_A;\Z)\lra \Z\lra 0.
\]
This immediately implies: 
\[
H_1(\T_A;\Z) \cong \Z\oplus \coker(A-I).
\]
Since $A\in SL_2(\Z)$, $A-I$ can be singular only if $A$ is parabolic with $\trace(A)=2$. But we are 
considering only parabolic bundles of the form 
$\T_{A(0,-n)}$, and $\trace(A(0,-n)) = \trace(-\left(\begin{smallmatrix}1&n\\0&1\end{smallmatrix}\right))=-2$, therefore in all our 
cases $\coker(A-I)$ is a torsion group. This shows that $b_1(Y)=1$ for each torus bundle given in the statement. Let $W$ be a closed regular neighborhood of the divisor 
corresponding to $\T_A$ and given by one of Lemmas~\ref{l:ell_sd},~\ref{l:par_sd} and~\ref{l:hyp_sd}. 
By construction we have $\del W= -\T_A$, and  the homology exact sequence of the pair $(W, \T_A)$ 
contains the exact sequence 
\[
\cdots \lra H_2(\T_A;\Z)\lra H_2(W;\Z)\stackrel{Q_\Ga}{\lra} H_2(W,\T_A;\Z)\lra H_1(\T_A;\Z)\lra H_1(W;\Z) \lra 0.
\]
From this sequence we deduce 
\[
H_1(\T_A;\Z) \cong \Z\oplus \coker(Q_\Ga), 
\]
and therefore $\coker(Q_\Ga)\cong\coker(A-I)$. Since $\coker(A-I)$ is a torsion group we conclude that 
$Q_\Ga$ is nonsingular, hence the map $H_2(\T_A;\Z)\to H_2(W;\Z)$ vanishes. This implies that if $[F]\in H_2(\T_A;\R)$ is the class carried by a torus fiber of the fibration $\T_A\to S^1$ and $i_* : H_2(\T_A;\R)\to H_2(X;\R)$ is the map induced by inclusion, we have $i_*([F])=0$. Therefore, when we evaluate on $[F]$ the restriction of $[\om_0]\in H^2(X;\R)$ to $H^2(\T_A;\R)$ we get
\[
\langle i^*[\om_0],[F]\rangle =  \langle [\om_0], i_*[F]\rangle = 0.
\]
Since $H_2(\T_A;\R)$ is generated by $[F]$ we conclude $i^*([\om_0])=0$, i.e.~the restriction of $\om_0$ to $\T_A$ is exact. Finally, $Q_\Ga$ is never negative definite, as one can easily check by looking at the corresponding intersection graph $\Ga$. 
We can therefore apply Theorem~\ref{t:limak} as explained at the beginning. Theorem~\ref{t:limak} implies that 
there is a one-parameter family of symplectic forms on $X$ which interpolates between the K\"ahler form $\om_0$ 
and a symplectic form $\om_1$, with the property that any neighborhood of $D$ contains an $\om_1$--concave neighborhood. 
Since $X$ is compact, a Moser-type argument produces a diffeomorphism $\phi: X \to X$ such that $\phi^*\om_1 = \om_0$. 
Pushing forward via $\phi$ the integrable complex structure $J_0$ compatible with $\om_0$ yields an integrable complex 
structure $J_1$ compatible with $\om_1$.  Setting $Y=\T_A$, we obtain a symplectic cap $W_Y$ from any $\om_1$--concave 
neighborhood of $D$. Moreover, the complement $X'$ in $X$ of the interior of $W_Y$, endowed with the complex structure $J_1$, 
is a strictly pseudo-convex surface  in the sense of~\cite{BdO97}. By~\cite{BdO97}*{Theorem~2'} there is a small deformation of $X'$ which is a Stein filling of $(Y,\xi_Y)$. This concludes the proof. 
\end{proof}

\section{Fillings}\label{s:fillings}

In this section we prove Theorem~\ref{t:main}. The theorem will follow combining 
Theorems~\ref{t:ell1-hyp},~\ref{t:ell2} and~\ref{t:par} below. At the end of the section we prove 
Proposition~\ref{p:distfill}, which shows that the family $\{(Y,\xi_Y)\ |\ Y\in\F\}$ contains infinitely many contact hyperbolic torus bundles admitting non-homotopy equivalent Stein fillings with even intersection forms and the same Betti numbers.

\begin{thm}\label{t:ell1-hyp}
Let $Y$ be a torus bundle over $S^1$ such that one of the following holds: 
\begin{enumerate} 
\item \label{ell1-hyp-item1}
$Y$ is elliptic and $Y=\T_{A(\ep)}$ with $\ep\in\{-1,0,+1\}$; 
\item \label{ell1-hyp-item2}
$Y$ is hyperbolic and $Y=-\T_{A(-c)}$ with $c\geq 3$; 
\item \label{ell1-hyp-item3}
$Y$ is hyperbolic and $Y = \T_{-A(d)}$ with $d$ embeddable.
\end{enumerate} 
Then, 
\begin{itemize}
\item 
each minimal, strongly convex symplectic filling of $(Y,\xi_Y)$ has vanishing first Chern class and first and third Betti numbers;
\item 
in Cases~\ref{ell1-hyp-item1} and~\ref{ell1-hyp-item2} the contact 3--manifold $(Y,\xi_Y)$ admits a unique 
minimal, strongly convex symplectic filling up to diffeomorphisms;
\item 
in Case~\ref{ell1-hyp-item3} all the minimal, strongly convex symplectic fillings of $(Y,\xi_Y)$ share 
the same second Betti number and fall into finitely many diffeomorphism classes.
\end{itemize}
\end{thm}

\begin{proof} 
Cases~\ref{ell1-hyp-item1}, \ref{ell1-hyp-item2} and~\ref{ell1-hyp-item3} of the statement correspond respectively to Cases~\ref{caps-item1}, \ref{caps-item4} and~\ref{caps-item5} of  
Theorem~\ref{t:caps}. In all cases the dual intersection graph of the symplectic divisor $D\subset W_Y$ 
contains at least three vertices, one of which has weight $+1$. This latter vertex 
corresponds to an embedded symplectic sphere $S\subset W_Y$ with self-intersection $+1$. 

We first deal with Cases~\ref{ell1-hyp-item1} and~\ref{ell1-hyp-item3}. 
If we blow up symplectically $W_Y$ at a nodal point of 
$D$ away from $S$, the total transform $\widetilde D \subset\widehat W_Y := W_Y\#\CPb$ contains at 
least four spheres. Notice that the boundary of the symplectic cap $\widehat W_Y$ is still 
strongly $\omega$--concave and the contact structure induced on the boundary is still $\xi_Y$. 
Let $P$ be a minimal, strongly convex symplectic filling of $(Y,\xi_Y)$. 
Since the boundary of $\widehat W_Y$ is strongly $\omega$--concave, we can construct a closed 
symplectic 4--manifold $(X,\omega)$ by symplectically gluing the cap $\widehat W_Y$ and $P$ together 
along $Y$ after possibly rescaling the symplectic form on one of the two pieces. 
Since $\widetilde D\subset \widehat W_Y$, $X$ contains an embedded 
symplectic $+1$-sphere $S$. Hence, by~\cite{McD90}*{Theorem~1.1 and Corollary~1.6} $X$ is symplectomorphic 
to a symplectic blowup of $\CP$ endowed with the standard K\"ahler form, in such a way that $S$ represents the hyperplane class. Since $b_1(Y)=b_1(\widehat W_Y)=1$ and $b_1(X)=b_3(\widehat W_Y) = 0$, the Mayer--Vietoris exact sequence of homology groups associated to the decomposition $X=\widehat W_Y\cup P$ shows that $b_1(P)=b_3(P)=0$.
We can choose an $\om$--tame almost complex structure $J$ on $X$ which makes all the 
symplectic spheres in $\widetilde D$ pseudo-holomorphic. Let $S'\subset\widetilde D$ be one of the two symplectic 
spheres intersecting $S$. By construction $\widetilde D' :=\widetilde D\setminus S'$ 
consists of a chain of $k$ symplectic spheres for some $k\ge 3$; their self-intersection numbers are $(1,1-b_1,-b_2,\dots,-b_k)$.
By~\cite{Li08}*{Theorem 4.2} there is a sequence of symplectic blowdowns of $X$ to $\CP$ such that $\widetilde D'$ blows down to 
the union of two lines $\ell\cup\ell'\subset\CP$. Moreover, at each step the almost complex structure descends, 
and, since $P$ is minimal, the exceptional divisor that we blow down either intersects the configuration positively once or belongs 
to the configuration. During this process the sphere $S'$ blows down to a smoothly embedded symplectic sphere 
intersecting both $\ell$ and $\ell'$ exactly once, hence $S'$ blows down to a line. 
It follows that $\widetilde D$ blows down to a generic configuration $C$ of three generically embedded symplectic 
spheres which are pseudo-holomorphic with respect to an almost complex structure tamed by the standard K\"ahler 
form on $\CP$. By a theorem of Gromov~\cite{Gr85} (see also~\cite{St15}*{Lemma 2.7}) the embedding of 
such three symplectic spheres in general position is unique up to isotopy. Therefore, up to isotopy we may assume 
that $C$ coincides with a basic configuration $(3\ell)$ of three complex lines. This means that the configuration 
$\widetilde D$ is obtained from $(3\ell)$ via a sequence of blowups. Since the homology class carried by 
the divisor $(3\ell)$ is Poincar\'e dual to  $c_1(\CP)$, we conclude that $c_1(X)$ is Poincar\'e dual to $[\widetilde D]$. 
In particular, $c_1(P)=0$, the total number of blowups must be $N=9-[\widetilde D]^2$, and the second Betti number of $P$ 
is determined to be $b_2(P)=N+1-b_2(W_Y)$. 
The homology classes carried by the symplectic spheres comprising $\widetilde D$ are determined as 
in~\cite{Li08}*{Theorem 4.2}, up to a little proviso: 
one needs to pay attention to the way $S'$ intersects the other sphere $S''$ which intersects $S$ nontrivially. 
Since we made sure that $\widetilde D$ contains at least four spheres, $S'$ and $S''$ intersect trivially. 
This implies that if we denote by $h$ the hyperplane class and by $e_i$ the classes of the exceptional divisors, 
since both $[S']$ and $[S'']$ are of the form $h+\sum c_i e_i$ and there 
exists exactly one index $i$ such that the coefficient $c_i$ in both expressions is nonvanishing (and equal to $-1$). 
Indeed, one can check that the exceptional 
divisor corresponding to $e_i$ comes from blowing up two lines of $(3\ell)$ at their intersection point, 
and that there is a divisor in $\widetilde D$ carrying a class $e_i-\sum_{j\neq i} x_je_j$ 
for some $x_j\ge 0$. But there are clearly finitely many possible sequences of blowups 
compatible with the above construction, and exactly one (up to reordering) in Case~\ref{ell1-hyp-item1}. It follows 
that the diffeomorphism type of the complement of a neighborhood 
of $\widetilde D\hookrightarrow X\cong\CP\#N\CPb$ is uniquely determined in Case~\ref{ell1-hyp-item1}, 
and determined up to finitely many possiblities in Case~\ref{ell1-hyp-item3}. 
This concludes the proof in Cases~\ref{ell1-hyp-item1} and~\ref{ell1-hyp-item3}. 

The proof in Case~\ref{ell1-hyp-item2} is quite similar, so we just outline the differences with the previous cases. 
In this case 
we do not blow up $W_Y$ at the beginning, so we consider directly the closed symplectic 4--manifold 
$X = W_Y \cup P$, where $P$ is a minimal, strongly convex symplectic filling. By the same argument as above, 
$X$ is symplectomorphic to a blowup of $\CP$ and $b_1(P)=b_3(P)=0$. The symplectic divisor $D$ is a union of smoothly embedded 
symplectic spheres $S$ and $S'$, where $S\cdot S=+1$, $S'\cdot S' = -c_1+2$ with $c_1\geq 3$ and $[S]=h$, 
where $h$ is the hyperplane class of $X$. Moreover, the adjunction formula for $S'$ and the fact that $S\cdot S'=2$ imply $[S']=2h-\sum_i e_i$, 
where the classes $e_i$ are the exceptional classes. As before, this implies that $D$ blows down to a 
configuration of two symplectic spheres in $\CP$, one representing $h$ and the other $2h$. But the moduli 
space of smoothly embedded symplectic curves in the class $2h$ in $\CP$ is connected and each pair of points 
determines a unique pseudo-holomorphic line~\cite{Gr85}, hence up to isotopy we may assume that $D$ blows down to a  
basic configuration $(\ell\mathcal{C}_2)$. Since there is clearly a unique way (up to reordering) to blow up 
$(\ell\mathcal{C}_2)$ to get $D$, the diffeomorphism type of $P$ is uniquely determined. Since the homology class carried by $(\ell\mathcal{C}_2)$ is Poincar\'e dual to $c_1(\CP)$, we conclude as in Cases~\ref{ell1-hyp-item1} and~\ref{ell1-hyp-item3} that $c_1(X)$ is Poincar\'e dual to $[D]$ and $c_1(P)=0$.
\end{proof}

\begin{thm}\label{t:ell2}
Let $Y$ be an elliptic torus bundle over $S^1$ of the form $Y=\T_{-A(\ep)}$, with $\ep\in\{-1,0,1\}$. 
Then, all Stein fillings of $(Y,\xi_Y)$ have vanishing first Chern class and first Betti number, and they are pairwise orientation-preserving diffeomorphic.
\end{thm}

\begin{proof} 
Let $(P, J)$ be a Stein filling of $(Y,\xi_Y)$. We start by arguing that $c_1(P) = 0$. By Honda's 
classification~\cite{Ho00}, there is only one isotopy class of contact structures without Giroux torsion on 
an elliptic bundle $Y$. Since fillable contact structures have no Giroux torsion~\cite{Ga06}, both $\xi_Y$ 
and its conjugate $\overline{\xi_Y}$ belong to this isotopy class. Since $\overline{J}$ is another Stein 
structure on $P$ which fills $\overline{\xi_Y}$, applying~\cite{LM97}*{Theorem 1.2} we conclude 
$c_1(P) = 0$. 

The elliptic bundles $Y$ of type $\T_{-A(\ep)}$ are considered in Case~\ref{caps-item2} of Theorem~\ref{t:caps}, which 
says that $Y$ has a symplectic cap $W_Y$ and the corresponding divisor $D$ has intersection graph 
$\Ga$ dual to the graph on the right of Figure~\ref{f:ell_bundles}. Therefore, there are smoothly embedded symplectic 
spheres $S_1, S_2\subset W_Y$ with $S_1\cdot S_1 = -1$, $S_2\cdot S_2 = \ep-2$ and $S_1\cdot S_2=+2$. 
The exceptional symplectic sphere $S_1\subset W_Y$ allows us to write $W_Y= W'_Y\#\CPb$, 
where $W'_Y$ is a symplectic cap of $Y$ diffeomorphic to a closed neighborhood of an immersed 
nodal symplectic sphere $S'_2$ with self-intersection $2+\ep$. Moreover, it is easy to check that 
$c_1(W'_Y)=PD(S'_2)$. 

Let $X'$ be a closed symplectic 4--manifold obtained by gluing  the symplectic cap $W'_Y$ to $P$ 
along their common boundary. First of all we want to argue that $b^+_2(X')=1$. 
Smoothing the singularity of $S'_2$ we obtain a smoothly embedded  2--torus with 
self-intersection $2+\ep>0$ inside $X'$. But such a torus violates the adjunction inequality, 
which is known to hold for closed, symplectic 4--manifolds 
with $b_2^+>1$. Therefore we must have $b_2^+(X')=1$. 

Now we claim that $c_1(X') = PD(S'_2)\in H^2(X')$ (we are going to use $\Z$ coefficients 
throughout the proof). Observe that each of the cohomology classes $c_1(X')$ and 
$PD(S'_2)$ both restrict as $0$ to $H^2(P)$ and as $c_1(W'_A)$ to $H^2(W'_Y)$. Therefore, in order to 
show that they are equal it suffices to check that the map $H^2(X')\to H^2(P)\oplus H^2(W'_Y)$ appearing in 
the Mayer--Vietoris sequence for the decomposition $X'=P\cup W'_Y$ is injective. This follows from the fact 
that the restriction map $H^1(W'_Y)\to H^1(Y)$ is surjective. The latter is equivalent, by Poincar\' e duality 
and the homology exact sequence of the pair $(W'_Y, Y)$, to the fact that the map $H_2(Y)\to H_2(W'_Y)$ 
induced by inclusion is the zero map, which follows immediately from the fact that $S'_2\cdot S'_2\neq 0$. 
Therefore the claim is established. 

Observe that, if $\om$ is the symplectic form on $X'$, the claim implies 
\[
c_1(X')\cdot [\om] = \int_{S'_2}\om >0.
\]
Thus, we can apply Theorem~\cite{Li96}*{Theorem~B}, which says that if $(X,\om)$ is a closed, symplectic 4--manifold 
with $b_2^+(X)=1$ and $K_X\cdot[\omega]<0$ then  $X$ is either rational (i.e.~a blowup of $\CP$) or ruled, 
i.e.~a symplectic sphere bundle. We conclude that $X'$ is either rational or ruled, and we claim that 
$X'$ cannot be ruled. In fact, suppose the contrary, and let $B$ be the base. Observe that 
$\chi(X') = \chi(B)\chi(S^2) = 2\chi(B)$. Moreover, from the Mayer--Vietoris sequence of the decomposition
$X'=N\cup (X'\setminus N)$, where $N$ is a regular neighborhood of a fiber, it is easy to deduce that $1\leq b_2(X')\leq 2$. 
Since the class of a symplectic fiber is nontrivial and of square zero, this immediately implies $\si(X')=0$. Therefore 
we have $c_1(X')^2 = 3\si(X') +2\chi(X') = 4\chi(B)$, contradicting the fact that $c_1(X')^2 = 2+\ep$ with $\ep\in\{-1,0,1\}$.  
We conclude that $X'$ must be rational, i.e.~symplectomorphic to an $r$--fold blowup 
of $\CP$. This implies $c_1^2(X') = 9-r$, with $r\in\{6, 7, 8\}$, and therefore $c_1(X') = PD(S'_2)= 3h-e_1-\cdots - e_r$, 
where $h$ is the hyperplane class and the classes $e_i$ the exceptional classes. Arguing as in the proof of Theorem~\ref{t:ell1-hyp} 
we can deduce that $S'_2$ is the proper transform of an $r$--fold blowup of a nodal pseudo-holomorphic cubic in $\CP$. 
Since the moduli space of pseudo-holomorphic nodal cubics is connected~\cite{Fr05}*{Theorem~13}, it follows that 
the diffeomorphism type of $P$ is determined, and given by the complement of 
a neighborhood of the strict transform of a nodal holomorphic cubic in an $r$--fold blowup of $\CP$, with $r\in\{6, 7, 8\}$. Finally, using the fact that $b_1(W'_Y)=1$ and $b_1(X')=0$ and arguing as in the proof of Theorem~\ref{t:ell1-hyp} shows that $b_1(P)=0$. 
\end{proof}

\begin{lemma}\label{l:parabolic4}
Let $(X,\omega)$ be a closed, symplectic 4--manifold containing the configuration $\Sigma$ of two transverse symplectic spheres described by the plumbing of Figure~\ref{f:par_bundles} for $n=4$. If $X\setminus \Sigma$ is minimal, then either $X= \CP\#\CPb$ and $\Sigma$ is the strict transform of the configuration $(\ell{\mathcal C}_2)$ blown up at a generic point of the line, or $X = S^2\times S^2$ and $\Sigma$ is the union of $S^2\times\{*\}$ and the graph of a holomorphic map $S^2\to S^2$ of degree $2$. In both cases the first Chern class of $X$ vanishes on 
$X\setminus\Sigma$. 
\end{lemma}

\begin{proof}
Let $S_1, S_2$ be the two symplectic spheres of $\Sigma$, with $S_1\cdot S_1 = 0$, $S_1\cdot S_2=+2$ and 
$S_2\cdot S_2 = +4$. By~\cite{McD90}*{Corollary~1.5}, the pair $(X,S_2)$ is an $r$--fold blowup of either $(\CP,q)$ or 
$(S^2\times S^2,\Gamma)$, where $q$ is a conic, $\Gamma$ is the graph of a holomorphic map 
$S^2\to S^2$ of degree $2$, and the exceptional spheres in $X$ are all disjoint from $S_2$. 
Call $e_1,\dots, e_r$ the exceptional  homology classes. 

If $(X,S_2)$ is a blowup of $(\CP,q)$, let $h$ be the homology class of a complex line in $\CP$, so that $[S_2] = 2h$. 
The conditions $S_1\cdot S_2 = 2$ and $S_1\cdot S_1 = 0$ imply $[S_1] = h-\sum x_ie_i$, with exactly one index $i$ such that $x_i=1$, while $x_j=0$ for each $j\neq i$. By positivity of intersections~\cite{McD91} each exceptional sphere 
is disjoint from $\Sigma$, and since $X\setminus \Sigma$ is minimal this means that $r = i = 1$. Moreover, the Poincar\'e
dual of $c_1(X)$ equals $[S_1]+[S_2]$, and the statement is proved in this case.

If $(X,S_2)$ is a blowup of $(S^2\times S^2,\Gamma)$ then $[S_2]=2s+f$, where 
$s = [S^2\times\{*\}]$ and $f = [\{*\}\times S^2]$. 
We have $[S_1] = as + bf - \sum x_ie_i$, with $a,b,x_i\ge 0$ by positivity of intersections.
Imposing that $S_1\cdot S_2 = 2$ we obtain $a+2b = 2$, therefore either $(a,b)=(2,0)$ or $(a,b)=(0,1)$. 
Imposing that $S_1\cdot S_1 = 0$ we obtain that $x_i = 0$ for each $i$. Finally, the adjunction formula excludes the case $(a,b) = (2,0)$, hence $[S_1] = f$, and positivity of intersections implies that each exceptional sphere in $X$ is disjoint from $\Sigma$. Therefore, since $X\setminus \Sigma$ is minimal, in this case we have $r = 0$. As in the previous case 
$c_1(X) = [S_1]+[S_2]$, and the statement is proved.
\end{proof}

\begin{lemma}\label{l:greaterthan4}
For $n>4$ the configuration of two symplectic spheres described by the plumbing of Figure~\ref{f:par_bundles} does not embed in any closed, symplectic 4--manifold.
\end{lemma}

\begin{proof}
Suppose by contradiction that there exists a closed symplectic 4--manifold $(X^0,\omega^0)$ containing two embedded spheres $S^0_1$, $S^0_2$ of self-intersection $0$ and $n$ respectively, with $S^0_1\cdot S^0_2 = 2$. 
Let $(X,\om)$ be the symplectic 4--manifold obtained by blowing up $X^0$ at $n-4$ generic points of $S^0_2$. Let $e_1,\dots, e_{n-4}$ be the corresponding exceptional classes, and $S_1$, $S_2$ the proper transforms of $S^0_1$ and $S^0_2$,  respectively. Now $S_1\cdot S_1 = 0$, $S_2\cdot S_2=4$ and $S_1\cdot S_2 = 2$. 
Notice that $[S_2]\cdot e_1=1$, and therefore the homology class $[S_2]$ 
cannot be even. By~Lemma~\ref{l:parabolic4}, and since $[S_2]$ is not even, $(X,\omega)$ must be a blowup of $S^2\times S^2$ with $[S_1]=f$, $[S_2]=2s+f$, where $s = [S^2\times\{*\}]$ and $f = [\{*\}\times S^2]$
and all the exceptional spheres are disjoint from $S_1\cup S_2$. 
This contradicts the fact that $[S_2]\cdot e_1=1$.
\end{proof}

\begin{thm}\label{t:par}
Let $Y$ be a parabolic torus bundle over $S^1$ of the form $Y = -\T_{A(0,-n)}$ with $n\leq 4$. Then, all minimal, strongly convex symplectic fillings of $(Y,\xi_Y)$ have vanishing first Chern class and first and third Betti numbers, and 
second Betti number equal to $4-n$. Moreover, if $n<4$ they are pairwise orientation-preserving diffeomorphic, while if $n=4$ they fall into at most two diffeomorphism classes.
\end{thm}

\begin{proof}
Let $P$ be a strongly convex symplectic filling of $(Y,\xi_Y)$. 
Let $W_Y$ be the symplectic cap of Theorem~\ref{t:caps}, Case~\ref{caps-item3}, and let $X$ be the closed, symplectic 4--manifold obtained 
by gluing $W_Y$ and $P$ together. The spherical symplectic divisor $D$ is contained in $X$, and therefore $X$ 
contains a symplectic sphere $F$ having self-intersection $0$ and a symplectic sphere $C$ of self-intersection $n$, with 
$F\cdot C=2$. Thus, according to~\cite{McD90}*{Theorem~1.4 and Corollary~1.5}, 
$X$ is symplectomorphic to a symplectic blowup of a symplectic $S^2$--bundle $p: X_0\to B$, in such a way that $F$ is mapped to 
a fiber. Let $e_1,\dots,e_N\in H_2(X;\Z)$ be the exceptional classes. Recall that a basis of the group $H_2(X;\Z)$ is given by 
the classes $[S], [F], e_1,\ldots, e_\ell$, where $S$ is a section of $p$. 
Moreover, both $[S]$ and $[F]$ are orthogonal to the classes $e_i$ and $[S]\cdot [F]=1$. Therefore we have 
$[C] = 2[S]+a[F]+\sum_i x_i e_i$ for some $a, x_i\in\Z$. 
We now claim that the base $B$ of the fibration has genus $g=0$.
In fact, suppose by contradiction that $g>0$. Then, there exist  
$\al, \be\in H^1(B;\R)$ such that $\langle \al\cup\be, [B]\rangle \neq 0$. Viewing 
$H^2(X_0;\R)$ as the subspace of $H^2(X;\R)$ consisting of those classes which vanish on 
$e_1,\ldots, e_\ell$, we have  
\begin{align*}
\langle p^*(\al)\cup p^*(\be), [C]\rangle &= 
\langle p^*(\al)\cup p^*(\be), 2[S]+a[F]\rangle = \\
&= \langle \al\cup\be, p_*(2[S]+a[F])\rangle 
= 2 \langle \al\cup\be, [B]\rangle \neq 0.
\end{align*}
On the other hand, since $C$ is a sphere the group $H^1(C;\R)$ vanishes, therefore 
\[
\langle p^*(\al)\cup p^*(\be), [C]\rangle = 0.
\]
This contradiction shows that $g=0$. That is, $X_0$ fibers over $\CPl$. Since there are two symplectic fibrations over $\CPl$ up to symplectomorphism, this means that $X$ is either a blowup of $\CP\#\CPb$ or a blowup of $S^2\times S^2$. 
In the first case $[F]=h-e_1$ and $[C] = ah-\sum b_i e_i$, where $h$ is the class of a line in $\CP$ and $a>0$, $b_i\ge 0$ by positivity of intersections~\cite{McD91}. In the second case $[F] = f$ and $[C]=as+bf-\sum c_ie_i$, where 
$s = [S^2\times\{*\}]$, $f = [\{*\}\times S^2]$ and by positivity of intersections $a, b, c_i\geq 0$. Notice that in both cases we have $b_1(X)=0$. Since $b_1(W_Y)=b_1(Y)=1$, the same Mayer--Vietoris argument used in the proof of Theorem~\ref{t:ell1-hyp} shows that $b_1(P)=b_3(P)=0$.
When $n=4$ the statement follows directly from Lemma~\ref{l:parabolic4}, therefore from now on we assume 
$n<4$. Our strategy will be to reduce the case $n < 4$ to the case $n=4$. 

We first analyze the case when $X$ is a blowup of $\CP\#\CPb$. We have three equations satisfied by $n$, $a$ and 
the numbers $b_i$. The first one comes from the self-intersection of $C$, the second one from the adjunction formula 
and the third one from the fact that $C$ intersects $F$ twice. They are given, respectively, by: 
\begin{equation}\label{e:par_system}
\left\{
\begin{array}{l}
n = a^2 - \sum b_i^2,\\
3a-\sum b_i = n+2,\\
a-b_1 = 2
\end{array}
\right.
\end{equation}
Subtracting the third equation from the second in~\eqref{e:par_system} we obtain
\begin{equation}\label{e:par_eq1}
2a-\sum_{i>1} b_i = n.
\end{equation}
The third equation in~\eqref{e:par_system} implies that $a^2-b_1^2 = a^2-(a-2)^2 = 4a-4$. 
Substituting this into the first equation in~\eqref{e:par_system} we get
\begin{equation}\label{e:par_eq2}
4a-4 - \sum_{i>1}\ b_i^2 = n.
\end{equation}
Now we subtract twice Equation~\eqref{e:par_eq1} from Equation~\eqref{e:par_eq2}, obtaining
\[ 
\sum_{i>1} (2b_i - b_i^2) = 4-n.
\]
Since $n\le 4$, the sum on the left-hand side must be nonnegative. For each index $i$ we have $2b_i-b_i^2\leq 1$, 
and equality holds if and only if $b_i=1$, therefore there must be at least $4-n$ indices $i>1$ such that $b_i = 1$. 
If there were $m>4-n$ indices $i_1,\dots i_m$ such that $b_{i_1} = \dots = b_{i_m} =1$, by blowing down the corresponding exceptional spheres we would obtain a configuration $C' \cup F$ of symplectic spheres in a 
closed symplectic 4--manifold contradicting Lemma~\ref{l:greaterthan4}. Therefore there must be exactly $4-n$ indices $i>1$ for which $b_i = 1$. It follows that for all other indices $j$ we have $b_j\in\{0,2\}$. However, by positivity of intersections, $b_j = 0$ corresponds to an exceptional divisors $e_j$ disjoint from $C\cup F$, against our assumption of minimality on the filling $P$. On the other hand, if $b_j = 2$ and $j>1$ then the class $h-e_1-e_j$ is a represented by an exceptional sphere disjoint both from $C$ and $F$, again contradicting the minimality of $P$. 

We conclude that total number of of exceptional classes is exactly $N=5-n$, and that $a=b_1+2$ and 
$b_2=\dots=b_N = 1$. Substituting these values in the second equation of~\eqref{e:par_system} we obtain $a=2$ and $b_1 = 0$. Summarizing, in this case $X$ is symplectomorphic to $\CP\#(5-n)\CPb$ and the 
spheres $F$ and $C$ are represented respectively by classes $h-e_1$ and $2h-e_2-\dots-e_{5-n}$. Moreover, 
from the Mayer--Vietoris sequence we get $b_2(P)=4-n$. 

Recall that in the second case, i.e.~when $X$ is a blowup of $S^2\times S^2$, we have $[F] = f$ and $[C]=as+bf-\sum c_ie_i$, where 
$s = [S^2\times\{*\}]$, $f = [\{*\}\times S^2]$ and by positivity of intersections $a, b, c_i\geq 0$. In fact, 
the minimality of $P$ implies $c_i>0$ for each $i$. Keeping in mind that the canonical class of $S^2\times S^2$ is 
Poincar\'e dual to $-2s-2f$, the analogues of Equations~\eqref{e:par_system} are 
\begin{equation}\label{e:par_system2}
\left\{
\begin{array}{l}
n =  2ab - \sum c_i^2,\\
2a+2b-\sum c_i = n+2,\\
a = 2.
\end{array}
\right.
\end{equation}
Manipulating the equations as in the previous case we obtain
\[
\sum_{i\ge 1} (2c_i-c_i^2) = 4-n,
\]
from which we deduce that $c_i\in\{1,2\}$ for each $i$. Finally, we observe that if $c_i=2$, the class 
$f-e_i$ is represented  by an exceptional sphere disjoint from $C\cup F$, contradicting the minimality of $P$.
Therefore $c_i=1$ for each $i$ and $b=1$. We conclude that $X$ is symplectomorphic to $(S^2\times S^2)\#(4-n)\CPb$ 
and the classes of $F$ of $C$ are given by $f$ and $2s+f-e_1-\dots-e_{4-n}$ respectively. As before, the 
Mayer--Vietoris sequence yields $b_2(P)=4-n$. 

As in the proof of Theorem~\ref{t:ell1-hyp}, up to isotopy we may assume that $D=F\cup C$ is the strict transform of a configuration $(\ell\mathcal{C}_2)$ or $S^2\times\{*\}\cup\Gamma$, each of which carries a homology class Poincar\'e dual to $c_1(X)$. Therefore, in each of the two cases the complement $K$ in $X$ of a regular neighborhood of the configuration $D$ is determined up to diffeomorphisms and the restriction of $c_1(X)$ to $K$
vanishes. This implies that the symplectic filling $P$ belongs to one of the two diffeomorphism 
classes above and that $c_1(P)=0$. 

In order to finish the proof it suffices to show that if $n<4$ the complements of regular neighborhoods of the 
configuration in the two cases are diffeomorphic. 

Observe that $\CP\#2\CPb$ contains an exceptional sphere $R$ representing the characteristic class $h-e_1-e_2$
and a symplectic sphere $T$ with $[T]=h-e_1$ and $T\cap R=\emptyset$. This implies that $\CP\#2\CPb$ 
is a symplectic blowup of a a spin, symplectic 4--manifold $Z$ containing a symplectic sphere of square zero. By 
the results of~\cite{McD90}, $Z$ is diffeomorphic to $S^2\times S^2$. This shows that there is a symplectomorphism 
$\psi:\CP\#2\CPb \to (S^2\times S^2)\#\CPb$, sending the class $h-e_1-e_2$ to the exceptional class $e$ and the class $h-e_1$ to $f$. It is easy to check that $\psi$ must also send the homology class $2h-e_2$ 
to the homology class $2s+f-e$. Gromov's results~\cite{Gr85} imply that, up to isotopy, $\psi$ maps the strict transform of a line representing the class $h-e_1$ to the strict trasforn of a sphere $\{*\}\times S^2$ representing the class $f$, 
and the strict transform of a conic representing the class $2h-e_2$ to the strict transform of a  
graph $\Gamma$ representing the homology class $2s+f-e$. 
Clearly, for each $m\geq 2$ there is a symplectomorphism between $\CP\# m\CPb$ and $S^2\times S^2\# (m-1)\CPb$ 
with the same properties. This shows that the complements of regular neighborhoods of the configuration in the two 
cases are diffeomorphic to each other, and concludes the proof.   
\end{proof}

In view of Theorem~\ref{t:ell1-hyp}, it is natural to wonder how many diffeomorphism types of strongly convex, minimal  
symplectic fillings a given contact hyperbolic torus bundle $(Y, \xi_Y)$ may have. We do not 
answer this question in general, but we are able to establish the following result.

\begin{prop}\label{p:distfill}
There exist infinitely many contact hyperbolic torus bundles $(Y, \xi_Y)$ admitting non-homotopy equivalent 
Stein fillings. 
\end{prop} 

\begin{proof}
Let us denote by $\Ga(1,1-c_1, -c_2,\ldots, -c_{\ell-1}, 1-c_\ell)$ the graph on the left of Figure~\ref{f:hyp_bundles2}. Inside a blowup 
$\CP\#\overline{\CP}$ of the standard K\"ahler $\CP$ we can easily find a spherical complex divisor $D(1,0,-1,0)$ 
having dual intersection graph $\Ga(1,0,-1,0)$ and whose complex spheres represent homology classes $h$, $h-e_1$, 
$e_1$ and $h-e_1$, where $h$ is the hyperplane class and $e_1$ is the exceptional class. Blowing up at the 
appropriate nodal point of the divisor and taking its proper trasform we get a spherical complex divisor $D(1,-1,-1,-2,0)$ 
inside $\CP\#2\overline{\CP}$ with dual intersection graph $\Ga(1, -1,-1,-2,0)$. Blowing up again we get a divisor 
$D(1,-1,-2,-1,-3,0)$ inside $\CP\#3\overline{\CP}$, whose spheres represent the classes $h$, $h-e_1-e_2$, $e_2-e_3$, 
$e_3$, $e_1-e_2-e_3$ and $h-e_1$, where the classes $e_i$ are the exceptional classes. Now we blow up in two different ways. 
First we blow up in such a way as to obtain a divisor $D_1=D(1,-2,-1,-3,-1,-3,0)\subset\CP\#4\overline{\CP}$, with spheres representing the classes: 
\[
h,\quad h-e_1-e_2-e_4,\quad e_4,\quad e_2-e_3-e_4,\quad e_3,\quad e_1-e_2-e_3\quad \text{and}\quad  h-e_1.
\]
Then, we blow up so as to obtain a divisor $D_2=D(1,-1,-3,-1,-2,-3,0)\subset\CP\#4\overline{\CP}$, with spheres representing the classes: 
\[
h,\quad h-e_1-e_2,\quad e_2-e_3-e_4,\quad e_4,\quad e_3-e_4,\quad e_1-e_2-e_3\quad \text{and}\quad  h-e_1.
\]
Finally, we suitably blow up another five times at smooth points of both $D_1$ and $D_2$, so that upon taking proper 
transforms we obtain two distinct divisors $D_1^{(0)}$ and $D_2^{(0)}$ inside $\CP\#9\overline{\CP}$ with the same 
dual intersection graph $\Ga(1,-2,-3,-3,-2,-3,-2)$. 

Let $P_i^{(0)}$, for $i=1,2$, denote the closure of a regular neighborhood $W_i^{(0)}$ of $D_i^{(0)}$. 
We claim that $P_1^{(0)}$ and $P_2^{(0)}$ are not homotopy equivalent. For the rest of this proof, whenever we mention 
homology groups we shall always implicitly use integer coefficients. Using the definition of $W_i^{(0)}$ one can check that the second homology group $H_2(W_i^{(0)})$ is free of rank 7. Let $X$ be the complex projective plane blown up nine times. Using the Mayer--Vietoris sequence for the decomposition $X = W_i^{(0)} \cup P_i^{(0)}$ one can check that $H_2(P_i^{(0)})$ is free Abelian of rank 4, and that its image $j_*(H_2(P_i^{(0)}))$ under the map induced by the inclusion $j : P_i^{(0)}\to X$ is isometric, as an intersection lattice, to $H_2(P_i^{(0)})/\langle T\rangle$, where $\langle T\rangle$ denotes the free, rank-1 subgroup generated by the class of the torus fiber in the boundary (pushed in the interior), which coincides with the kernel of the intersection pairing on $H_2(P_i^{(0)})$. This implies that the isometry class of the intersection lattice $j_*(H_2(P_i^{(0)}))$ is determined by the homotopy type of $P_i^{(0)}$. We claim that $j_*(H_2(P_1^{(0)}))$ and $j_*(H_2(P_2^{(0)}))$ are not isometric to each other, which in turn implies that $P_1^{(0)}$ and $P_2^{(0)}$ are not homotopy equivalent. To prove the claim we shall use the fact that, for $i=1,2$, $j_*(H_2(P_i^{(0)}))$ is isometric to the lattice $\Lambda_i^{(0)}$ orthogonal to the image of $H_2(W_i)$ under the map induced by the inclusion $W_i\subset X$. To see this fact one may observe that, given any $a\in H_2(X)$ orthogonal to all the homology classes of the spheres belonging to the divisor $D_i\subset W_i$, one can represent $a$ by a smooth, oriented surface disjoint from $D_i$, and therefore $a\in  j_*(H_2(P_i))$.

We now set out to compute the determinant of $\Lambda_i^{(0)}$. The classes of the spheres of $D_1^{(0)}$ are 
\[
h,\ h-e_1-e_2-e_4,\ e_4-e_5-e_6,\ e_2-e_3-e_4,\ e_3-e_7,\ e_1-e_2-e_3,\  h-e_1-e_8-e_9, 
\]
while the classes of the spheres of $D_2^{(0)}$ are 
\[
h,\ h-e_1-e_2-e_5,\ e_2-e_3-e_4,\ e_4-e_6-e_7,\ e_3-e_4,\ e_1-e_2-e_3,\  h-e_1-e_8-e_9.
\]
A direct calculation shows that the sublattice $\La_1^{(0)}$ of $H_2(\CP\#9\overline{\CP})$ orthogonal to the 
classes of $D_1^{(0)}$ has integral basis:
\[
\al_1 = e_5-e_6,\quad \al_2 = e_1 + e_3 - e_4 - e_5 + e_7 - e_8,\quad \al_3 = e_8 - e_9.
\]
On the other hand, the sublattice $\La_2^{(0)}$ orthogonal to the classes of $D_1^{(0)}$ 
has integral basis:
\[
\be_1 = e_6-e_7,\quad \be_2 = -3e_1 - 2e_2 - e_3 - e_4 + 5e_5 - e_6 + 3e_9,\quad \be_3 = e_8 - e_9.
\]
This shows that the lattices $\La_1^{(0)}$ and $\La_2^{(0)}$ are both even. The intersection matrix $(\al_i\cdot\al_j)$ has determinant $-20$, while the intersection matrix $(\be_i\cdot\be_j)$ has determinant $-180$, therefore $\La_1$ and $\La_2$ are not isometric, and $P_1^{(0)}$ and $P_2^{(0)}$ are not homotopy equivalent. Moreover, applying Theorem~\ref{t:limak} and the 
results of~\cite{BdO97} as in the proof of Theorem~\ref{t:caps} shows that $P_1^{(0)}$ and $P_2^{(0)}$ can be endowed with structures of Stein fillings of $(-\T_{-A(3,3,3,2,3,3)}, \xi_{-\T_{-A(3,3,3,2,3,3)}})$. 

This example belongs to an infinite family of examples obtained as follows.  
We blow up at $N\geq 1$ generic points of the sphere of $D_1^{(0)}$ representing $e_2-e_3-e_4$, and at $N$ 
generic points of the sphere of $D_2^{(0)}$ representing $e_4-e_6-e_7$. Taking proper transforms we get 
spherical complex divisors $D_i^{(N)}\subset\CP\#(N+9)\overline{\CP}$ having dual intersection graphs 
$\Ga(-1,-2,-3,-3-N,-2,-3,-2)$ and determining Stein fillings $P_i^{(N)}$, $i=1,2$. Arguing as for $D_i^{(0)}$, we get 
that the orthogonal lattice $\La_1^{(N)}$ has integral basis 
\[
\al_1,\quad\al_2,\quad\al_3,\quad \al_4 = e_9 - e_{10},\quad \al_5 = e_{10}  - e_{11},\quad\ldots\quad \al_{N+6} = e_{N+8}  - e_{N+9}, 
\]
while $\La_2^{(N)}$ has integral basis 
\[
\be_1,\quad\be_2,\quad\be_3,\quad \be_4 = e_9 - e_{10},\quad \be_5 = e_{10}  - e_{11},\quad\ldots\quad \be_{N+6} = e_{N+8}  - e_{N+9}.
\]
Then, an inductive computation yields
\[
\det(\al_i\cdot\al_j) = (-1)^{N+1}(9N+20)\quad\text{and}\quad \det(\be_i\cdot\be_j) = (-1)^{N+1} 9(9N+20).
\]
This shows that $P_1^{(N)}$ and $P_2^{(N)}$ are non-homotopy equivalent and carry structures of Stein fillings of the same 
contact hyperbolic bundle for each $N\geq 0$.  
\end{proof}

\section{Identifying the contact structures}\label{s:contstr}

In this section we use Honda's classification~\cite{Ho00} of tight contact structures on torus bundles over the circle 
(see also~\cite{Gi00}) to identify the contact structures $\xi_Y$ for elliptic bundles of the form 
$Y=\T_{-A(\ep)}$, with $\ep\in\{-1,0,1\}$ and $Y=\T_{A(1)}$, as well as 
for the hyperbolic bundles of Theorem~\ref{t:main}(4). We also give explicit constructions of Stein fillings for 
$(Y,\xi_Y)$ when $Y$ is elliptic as above. 

\begin{prop}\label{p:ell-class}
Let $Y$ be an elliptic torus bundle of the form $Y=\T_{-A(\ep)}$, with $\ep\in\{-1,0,1\}$ or $Y=\T_{A(1)}$. 
Then, the contact structure $\xi_Y$ is the unique tight contact structure on $Y$ 
with vanishing Giroux torsion. Moreover, $\xi_Y$ is universally tight. 
\end{prop}

\begin{proof} 
The bundles in question are associated to the monodromies 
\begin{equation*}
-A(-1) = \left(\begin{smallmatrix} 1 & -1\\1 & 0 \end{smallmatrix}\right),\quad 
-A(0) = \left(\begin{smallmatrix} 0 & -1\\1 & 0 \end{smallmatrix}\right),\quad
-A(1) = \left(\begin{smallmatrix} -1 & -1\\1 & 0 \end{smallmatrix}\right)\quad\text{and}\quad
A(1) = \left(\begin{smallmatrix} 1 & 1\\ -1 & 0 \end{smallmatrix}\right).
\end{equation*}
Defining $S=A(0)$ and $T=\left(\begin{smallmatrix} 1 & 1\\ 0 & 1 \end{smallmatrix}\right)$, it is easy to check that 
the first monodromy is conjugate to $-(T^{-1}S)^2$, the second and third ones are equal, 
respectively, to $-S$ and $-T^{-1}S$ and the last one to $T^{-1} S$. Then, Honda's classification~\cite{Ho00} 
implies that on the associated bundles there is only one isotopy class of tight contact structures without Giroux torsion, 
and that this isotopy class is universally tight (there are no virtually overtwisted contact structures on these bundles). Since fillable contact 
structures have no Giroux torsion~\cite{Ga06}, the contact structure $\xi_Y$ must be isotopic to the unique tight contact structure on $Y$ 
without Giroux torsion. 
\end{proof}

It might be interesting to see an explicit construction of a Stein filling of $(Y,\xi_Y)$ for the bundles 
of Proposition~\ref{p:ell-class}. It follows from the proposition and the fact that fillable contact structures have 
no Giroux torsion~\cite{Ga06} that $\xi_Y$ is the unique Stein fillable contact structure on $Y$. 
Therefore, in order to exhibit a Stein filling of $(Y,\xi_Y)$ it suffices to construct a single Stein 
4--manifold with boundary $X$ such that $\del X = Y$. Starting from the obvious Kirby 
diagrams corresponding to the graphs of Figures~\ref{f:ell_bundles} and~\ref{f:par_bundles} and using Kirby calculus 
it is a simple matter to check that each of the torus bundles $\T_{-A(\ep)}$, $\ep\in\{-1,0,1\}$, is the boundary of 
the 4--dimensional plumbing given in Figure~\ref{f:steinplumb-ell}. 
\begin{figure}[ht]
\centering
$\T_{-A(-1)} :$
\hspace{1cm}
$\xygraph{
!{<0cm,0cm>;<1cm,0cm>:<0cm,1cm>::}
!{(1,0) }*+{\bullet}="b"
!{(2,0) }*+{\bullet}="c"
!{(3,0) }*+{\bullet}="d"
!{(4,0) }*+{\bullet}="e"
!{(5,0) }*+{\bullet}="f"
!{(6,0) }*+{\bullet}="g"
!{(7,0) }*+{\bullet}="h"
!{(8,0) }*+{\bullet}="i"
!{(6,1) }*+{\bullet}="l"
!{(1,-0.4) }*+{-2}
!{(2,-0.4) }*+{-2}
!{(3,-0.4) }*+{-2}
!{(4,-0.4) }*+{-2}
!{(5,-0.4) }*+{-2}
!{(6,-0.4) }*+{-2}
!{(7,-0.4) }*+{-2}
!{(8,-0.4) }*+{-2}
!{(6.4,1) }*+{-2}
"b"-"c"
"c"-"d"
"d"-"e"
"e"-"f"
"f"-"g"
"g"-"h"
"h"-"i"
"g"-"l"
}$\\
\vskip 0.5cm
$\T_{-A(0)} :$
\hspace{2cm} 
$\xygraph{
!{<0cm,0cm>;<1cm,0cm>:<0cm,1cm>::}
!{(0,0) }*+{\bullet}="a"
!{(1,0) }*+{\bullet}="b"
!{(2,0) }*+{\bullet}="c"
!{(3,0) }*+{\bullet}="d"
!{(4,0) }*+{\bullet}="e"
!{(5,0) }*+{\bullet}="f"
!{(6,0) }*+{\bullet}="g"
!{(3,1) }*+{\bullet}="h"
!{(0,-0.4) }*+{-2}
!{(1,-0.4) }*+{-2}
!{(2,-0.4) }*+{-2}
!{(3,-0.4) }*+{-2}
!{(4,-0.4) }*+{-2}
!{(5,-0.4) }*+{-2}
!{(6,-0.4) }*+{-2}
!{(3.4,1) }*+{-2}
"a"-"b"
"b"-"c"
"c"-"d"
"d"-"e"
"e"-"f"
"f"-"g"
"d"-"h"
}$\\
\vskip 0.5cm
$\T_{-A(1)} :$
\hspace{3.8cm} 
$\xygraph{
!{<0cm,0cm>;<1cm,0cm>:<0cm,1cm>::}
!{(0,0) }*+{\bullet}="a"
!{(1,0) }*+{\bullet}="b"
!{(2,0) }*+{\bullet}="c"
!{(3,0) }*+{\bullet}="d"
!{(4,0) }*+{\bullet}="e"
!{(2,1) }*+{\bullet}="f"
!{(2,2) }*+{\bullet}="g"
!{(0,-0.4) }*+{-2}
!{(1,-0.4) }*+{-2}
!{(2,-0.4) }*+{-2}
!{(3,-0.4) }*+{-2}
!{(4,-0.4) }*+{-2}
!{(2.4,1) }*+{-2}
!{(2.4,2) }*+{-2}
"a"-"b"
"b"-"c"
"c"-"d"
"d"-"e"
"c"-"f"
"f"-"g"
}$
\caption{Plumbings bounding the elliptic bundles $\T_{-A(\ep)}$, $\ep\in\{-1,0,1\}$.}
\label{f:steinplumb-ell}
\end{figure} 
In the same way it is easy to check that the bundle $\T_{A(1)}$ is the boundary of the smooth 
4--dimensional handlebody obtained by attaching a 4--dimensional 2--handle to the 4--ball along 
the right-handed trefoil knot in $S^3$ with framing $0$. Using e.g.~the results of~\cite{Go98} it is 
straightforward to check that each one of the smooth 4--manifolds just described carries a Stein 
structure with boundary and therefore gives a Stein filling of the corresponding torus bundle.

We now consider hyperbolic bundles. 
Let $Y$ by a hyperbolic torus bundle with $Y=\T_A$ and $\trace(A)< -2$. Then, as explained in Section~\ref{s:intro}, 
$Y=\T_{-A(d)}$, where $d=(d_m,d_{m-1},\ldots, d_1)$, 
$d_i\geq 2$ for all $i$ and $d_i\geq 3$ 
for some $i$. Moreover, by~\cite{Ne81}*{Theorem~6.1} $Y$ is the oriented boundary of the 4--dimensional 
plumbing $P_-(d)$ given by the diagram of Figure~\ref{f:plumb-hyp}. 
\begin{figure}[ht]
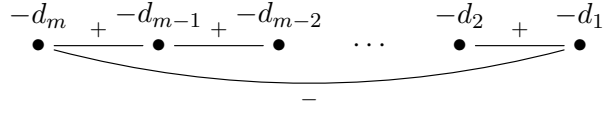

\centering
$\xygraph{
!{<0cm,0cm>;<1cm,0cm>:<0cm,1cm>::}
!{(0,1) }*+{\bullet}="c"
!{(1.6,1) }*+{\bullet}="d"
!{(3.2,1) }*+{\bullet}="e"
!{(5.6,1) }*+{\bullet}="f"
!{(7.2,1) }*+{\bullet}="g"
!{(0,1.4) }*+{-d_m}
!{(1.6,1.4) }*+{-d_{m-1}}
!{(3.2,1.4) }*+{-d_{m-2}}
!{(5.6,1.4) }*+{-d_2}
!{(7.2,1.4) }*+{-d_1}
!{(4.4,1)}*+{\cdots}
"c"-"d"^+
"d"-"e"^+
"f"-"g"^+
"g"-@/^0.5cm/"c"^{-}
}$
\caption{Plumbings $P_-(d)$ bounding the hyperbolic bundles $\T_{-A(d)}$.}
\label{f:plumb-hyp}
\end{figure} 

\begin{lemma}[\cites{Ho00, Gi00}]\label{l:vot}
$\T_{-A(d)}$ carries $(d_1-1)(d_2-1)\cdots(d_m-1)$ tight, virtually overtwisted 
contact structures up to isotopy and one universally tight contact structure with no Giroux torsion up to 
contactomorphisms.
\end{lemma} 

\begin{proof} 
One can easily check that $-A(d) = -T^{-d_1}ST^{-d_2}S\cdots T^{-d_{m-1}}ST^{-d_m}S$, where $S=A(0)$ and 
$T=\left(\begin{smallmatrix} 1 & 1\\ 0 & 1 \end{smallmatrix}\right)$. The results of~\cites{Ho00, Gi00} are 
written in terms of such a factorization of $-A(d)$, and they immediately imply the statement (see e.g.~the table on 
page 90 of~\cite{Ho00}).
\end{proof} 

\begin{lemma}\label{l:VOT_filling}
Let $Y$ by a hyperbolic torus bundle with $Y=\T_{-A(d)}$. Then, each virtually overtwisted tight contact 
structure on $Y$ admits a Stein filling $P$ with $b_1(P)=1$. 
\end{lemma}

\begin{proof}
Figure~\ref{f:VOT_filling} represents the front diagram of a Legendrian link $\bL$ inside the standard contact $S^2\times S^1$ 
viewed as $S^3$ with a 3--dimensional 1--handle attached (see e.g.~\cite{Go98}). Recall that there is no need to specify the 
over/under information at each crossing of the diagram, because the over-strand is always the one with smaller slope. 
Label the components of the link $\bL$ as $L_1,\dots, L_m$ from left to right, with $L_m$ being the component going 
over the 1--handle. Orient the components so that $\lk(L_m,L_1) = -1$ and $\lk(L_j,L_{j+1}) = 1$ for any 
$1\le j < m$, and Legendrian stabilize $d_j-2$ times each component. For each index $j$, $L_j$ becomes a Legendrian 
unknot (which we keep denoting $L_j$) with $\tb(L_j) = -d_j+1$. Since there are $d_j-1$ ways to stabilize an 
oriented Legendrian knot $d_j-2$ times, for each $j$ we get $d_j-1$ isotopy classes of such unknots, 
distinguished by their rotation numbers. We can then attach a 4--dimensional Stein handle along each $L_j$ 
to $S^2\times S^1$ viewed as the boundary of $B^4$ union a 4--dimensional 1--handle, obtaining $\prod_j (d_j-1)$ 
Stein structures with boundary on the smooth 4--dimensional plumbing $P_-(d)$.  
Notice that $\pi_1(P_-(d)) = \Z$, and in particular $b_1(P_-(d))=1$. 
\begin{figure}[ht]
\includegraphics[width=9cm]{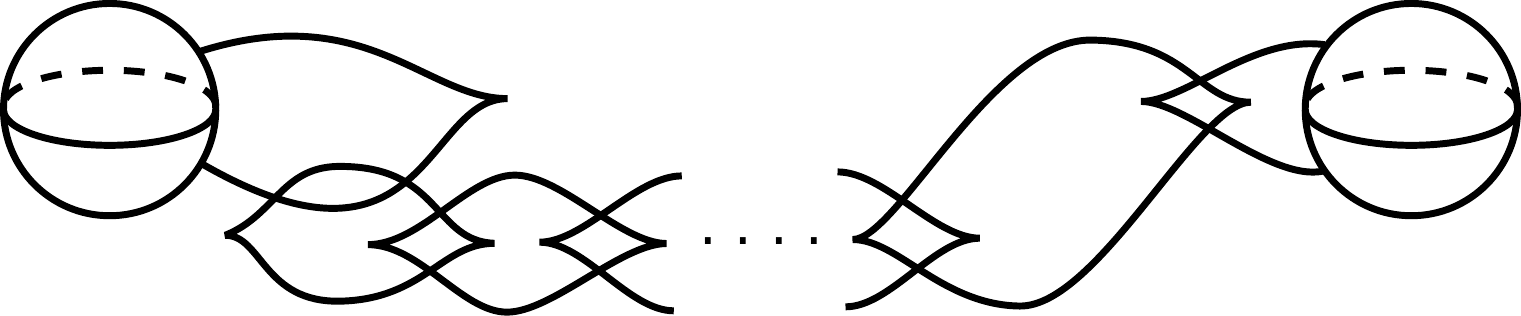}
\caption{The Legendrian link $\bL$}\label{f:VOT_filling}
\end{figure}
We now want to argue that the $\prod_j (d_j-1)$ Stein 4--manifolds we just described yield distinct isotopy classes of contact structures on 
$\del P_-(d)=\T_{-A(d)}$. Fix a choice of stabilizations, or equivalently the $m$-tuple $r = (r_1,\dots,r_m) = (\rot(L_1),\dots, \rot(L_m))$ of rotation 
numbers of the components of $\bL$, and let $J_r, \xi_r$ be the associated Stein structure on $P_-(d)$ and induced contact structure on 
$\T_{-A(d)}$, respectively. Also, let $S_j$ be the homology class carried by the oriented $j$-th sphere of the plumbing obtained by capping 
off an oriented disk bounding $L_j$ with the core of the attached 4--dimensional 2--handle. Since $\langle c_1(J_r), S_j\rangle = r_j$, 
by~\cite{LM97}*{Theorem 1.2} the contact structures $\xi_r, \xi_{r'}$ corresponding to the two $m$-tuples $r, r'$ can be isotopic 
only if $r=r'$. Thus, the collection $\{\xi_r\}_r$ consists of $(d_1-1)(d_2-1)\cdots(d_m-1)$ pairwise nonisotopic contact structures 
on $\T_{-A(d)}$, each admitting a Stein filling $P$ with $b_1(P)=1$. In view of Lemma~\ref{l:vot}, in order to finish the 
proof it suffices to show that $\xi_r$ is virtually overtwisted for every $r$. 
Consider the double cover $\widetilde{P_-(d)} \to P_-(d)$ associated to the subgroup $2\Z \subset \Z = \pi_1(P_-(d))$. 
If we denote with $d'$ the string $(d,d)$, it is easy to check that $\widetilde{P_-(d)} = P_+(d')$, where $P_+(d')$ is the plumbing  
described by the diagram of Figure~\ref{f:P+}. 
\begin{figure}[ht]
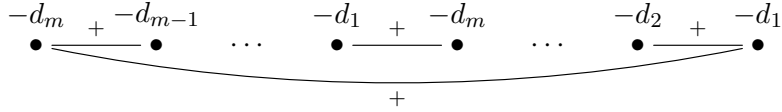

\centering
$\xygraph{
!{<0cm,0cm>;<1cm,0cm>:<0cm,1cm>::}
!{(0,1) }*+{\bullet}="c"
!{(1.6,1) }*+{\bullet}="d"
!{(4.0,1) }*+{\bullet}="x"
!{(5.6,1) }*+{\bullet}="y"
!{(8.0,1) }*+{\bullet}="f"
!{(9.6,1) }*+{\bullet}="g"
!{(0,1.4) }*+{-d_m}
!{(1.6,1.4) }*+{-d_{m-1}}
!{(4.0,1.4) }*+{-d_1}
!{(5.6,1.4) }*+{-d_m}
!{(8.0,1.4) }*+{-d_2}
!{(9.6,1.4) }*+{-d_1}
!{(2.8,1)}*+{\cdots}
!{(6.8,1)}*+{\cdots}
"c"-"d"^+
"f"-"g"^+
"x"-"y"^+
"g"-@/^0.5cm/"c"^{+}
}$
\caption{The 4--dimensional plumbing $P_+(d')$.}\label{f:P+}
\end{figure} 
Moreover, by the proof of~\cite{Ne81}*{Theorem~6.1} we have $\del{P_+(d')} = \T_{A(d')}$.

The Stein structure $J_r$ on $P_-(d)$ pulls back to a Stein structure $\widetilde{J_r}$ on $\widetilde{P_-(d)}$, 
obtained by attaching Stein handles along suitably oriented and stabilized components of the Legendrian link 
$\bL' =L'_1\cup\cdots\cup L'_{2m}$ of Figure~\ref{f:VOT_DC}. 
\begin{figure}[ht]
\centering
\includegraphics[width=12cm]{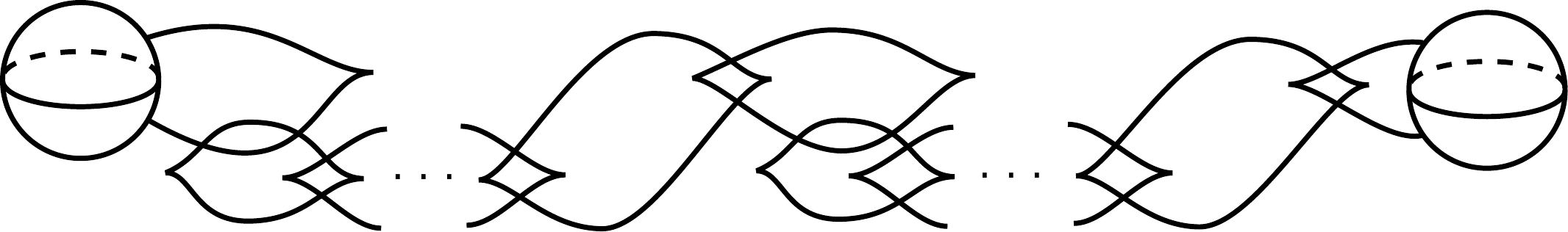}
\caption{The Legendrian link $\bL'$}\label{f:VOT_DC}
\end{figure}
With the orientations and stabilizations just described, $L'_i$ and $L'_{i+m}$ acquire the same 
rotation number $r_i$ for all $i$ (where it is understood that $L'_{2m+1} = L'_1$). 

We want to compare the Stein structures $J_r$ on $P_+(d')$ with the Stein structures constructed by Bhupal and Ozbagci 
on the same 4--manifold~\cite{BO14}. Let $S'_j\in H_2(P_+(d'))$, $j=1,\ldots, m$, be the homology class carried by the oriented spheres 
obtained by capping off an oriented disk bounding $L'_j$ with the core of the 4--dimensional 2--handle. Observe that $S'_j\cdot S'_{j+1} =+1$ 
for $j=1,\ldots, m-1$. Now we choose further homology classes $S'_j\in H_2(P_+(d'))$, $j=m+1,\ldots, 2m$, carried by spheres 
obtained by capping off disks as before, but we orient the spheres so that $S'_j\cdot S'_{j+1} = 1$ for every $i$ 
(where it is understood that $S'_{2m+1} = S'_1$). This implies that the classes $S'_{m+1},\dots, S'_{2m}$ are carried by spheres 
obtained by capping off oriented discs bounding $-L'_{m+1}, \dots, -L'_{2m}$. Therefore 
$\langle c_1(\tilde{J}_r),S'_j\rangle = r_j$ if $1\le j \le m$, while $\langle c_1(\tilde{J}_r),S_j\rangle = -r_j$ for $m< j \le 2m$.
On the other hand, according to~\cite{BO14}*{Proposition~11} for each universally tight contact structure on $\del{P_+(d')} = \T_{A(d')}$
there is a Stein structure $J$ on $P_+(d')$ such that $\langle c_1(J), S'_j \rangle=\ep (d_j-2)$ for each $j$, where $\ep = \pm 1$ is independent of $j$
(to check the orientations of the spheres see~\cite{BO14}*{Figure~6}). Applying~\cite{LM97} we conclude that, in order 
for the contact structure $\tilde{\xi}_r$ induced on $ \T_{A(d')}$ to be universally tight, we would need to have $d_j=2$ for each $j$. 
But this contradicts our assumption that $d_j\geq 3$ for at least one $j$. Therefore, each $\tilde{\xi}_r$ is virtually overtwisted, and so is $\xi_r$. This concludes the proof.
\end{proof}

We are now ready to state our result for hyperbolic bundles. 

\begin{thm}\label{t:hyp-class}
Let $Y$ by a hyperbolic torus bundle with $Y=\T_{-A(d)}$ with $d$ embeddable. Then, the contact structure $\xi_Y$ is the unique 
universally tight contact structure on $Y$ with vanishing Giroux torsion. 
\end{thm}

\begin{proof}
By Theorem~\ref{t:caps} the contact structure $\xi_Y$ is Stein fillable, and by~\cite{Ga06} it has no Giroux torsion. 
Suppose that $\xi_Y$ is virtually overtwisted. Then, by Lemma~\ref{l:VOT_filling}, $(Y,\xi_Y)$ admits a Stein filling $P$ with $b_1(P)=1$. But by Theorem~\ref{t:ell1-hyp} each Stein filling of $(Y,\xi_Y)$ has vanishing first Betti number. This shows that $\xi_Y$ is universally tight, and Honda's classification~\cite{Ho00} implies that on the underlying bundle there is only one isotopy class of universally tight contact structures without Giroux torsion (see e.g.~the table of~\cite{Ho00}*{Page 90}).
\end{proof}
\bibliographystyle{amsplain}
\bibliography{fillings}
\end{document}